\documentclass[12pt,reqno]{article}

\usepackage[usenames]{color}
\usepackage{amssymb}
\usepackage{amsmath}
\usepackage{amsthm}
\usepackage{amsfonts}
\usepackage{amscd}
\usepackage{graphicx}

\usepackage[colorlinks=true,
linkcolor=webgreen,
filecolor=webbrown,
citecolor=webgreen]{hyperref}

\definecolor{webgreen}{rgb}{0,.5,0}
\definecolor{webbrown}{rgb}{.6,0,0}

\usepackage{color}
\usepackage{fullpage}
\usepackage{float}

\usepackage{graphics}
\usepackage{latexsym}

\usepackage{xcolor}

\begin{document}

\theoremstyle{plain}
\newtheorem{theorem}{Theorem}
\newtheorem{proposition}[theorem]{Proposition}
\newtheorem{corollary}[theorem]{Corollary}
\newtheorem{lemma}[theorem]{Lemma}
\theoremstyle{definition}
\newtheorem{example}[theorem]{Example}
\newtheorem*{remark}{Remark}

\begin{center}
\vskip 1cm
{\LARGE\bf Finite sums associated with some polynomial identities \\ }

\vskip 1cm

{\large
Kunle Adegoke \\
Department of Physics and Engineering Physics \\ Obafemi Awolowo University \\ 220005 Ile-Ife, Nigeria \\
\href{mailto:adegoke00@gmail.com}{\tt adegoke00@gmail.com} \\

\vskip .25 in

Robert Frontczak \\
Independent Researcher \\ 72762 Reutlingen,  Germany \\
\href{mailto:robert.frontczak@web.de}{\tt robert.frontczak@web.de}

\vskip .25 in

Karol Gryszka \\
Institute of Mathematics \\ University of the National Education Commission, Krakow \\
Podchor\c{a}\.{z}ych 2, 30-084 Krak{\'o}w, Poland\\
\href{mailto:email}{\tt karol.gryszka@uken.krakow.pl} 
}

\end{center}

\vskip .2 in

\begin{abstract}
In this paper, we present a general framework for the derivation of interesting finite combinatorial sums starting with certain classes of polynomial identities. The sums that can be derived involve products of binomial coefficients and also harmonic numbers and squared harmonic numbers. We apply the framework to discuss combinatorial sums associated with some prominent polynomial identities from the recent past. 
\end{abstract}



\section{Motivation and Preliminaries}

Harmonic numbers $(H_n)_{n\geq 0}$ are defined by
\begin{equation*}
H_0 = 0 \quad \mbox{and for $n\geq 1:$} \qquad H_n = \sum_{k=1}^n \frac{1}{k}.
\end{equation*}
Due to their ubiquitous connections with combinatorics and number theory harmonic numbers 
and generalized harmonic numbers are very popular research objects. The mathematical literature is extremely rich, 
still growing constantly (see the papers of Batir \cite{1, Batir2020,2,3}, Boyadzhiev \cite{4,5}, Choi \cite{6, Choi13, 7}, 
Chu \cite{8,Chu12,9}, Gen\v{c}ev \cite{13}, Jin and Du \cite{15}, Sofo and Srivastava \cite{17}, Wang and Wei \cite{20}, 
Wei et al.\ \cite{21}, Wei and Wang \cite{22} or Xu \cite{23}, to name a few recent and important articles in the subject).

In 2009, Boyadzhiev \cite{4} studied binomial sums with harmonic numbers using the Euler transform. 
His main result is the following identity valid for $n\geq 1$ and all $t\in\mathbb{C}$
\begin{equation}\label{Boyadzhiev}
\sum_{k=1}^{n} \binom{n}{k} H_{k} t^{k} = (1+t)^{n} \left ( H_{n} - \sum_{k=1}^{n}\frac{1}{k (1+t)^k} \right ).
\end{equation}
We have learned recently, however, that Boyadzhiev's identity \eqref{Boyadzhiev} was discovered earlier by Donald Knuth.
Knuth published it in his book \cite[Page 77, Eq. (10)]{Knuth} in 1997. So, Boyadzhiev actually only rediscovered this identity, 
which will be referred to as the Knuth-Boyadzhiev (KB) identity. Frontczak \cite{11} obtained a complement of the KB identity 
given in \eqref{Boyadzhiev} and derived
\begin{equation}\label{Frontczak}
\sum_{k=0}^{n} \binom{n}{k} H_{k} t^k = \left ((1+t)^n - 1 \right )H_n - t \sum_{k=0}^{n-1} H_{n-1-k} (1+t)^k.
\end{equation}
In addition, Frontczak \cite{12} also derived an analogue formula involving the skew-harmonic numbers. 
It is also notable that the KB identity is a special case of a binomial sum identity involving multiple 
harmonic-like numbers derived by Adegoke and Frontczak \cite{Adegoke}.
Batir \cite{2} and Batir and Sofo \cite{3} did research in the same direction. Among other things, they derived in \cite{3} the identity
\begin{equation} \label{BaSo}
\sum_{k=1}^{n} \binom{n}{k} A_k t^k = (1+t)^{n-1}\sum_{k=0}^{n-1} \left ( \sum_{j=0}^{k} \binom{k}{j} a_{j+1} t^{j+1}\right )\frac{1}{(1+t)^k},
\end{equation}
where $(a_n)_{n\geq 1}$ is any sequence in $\mathbb{C}$, and $A_n=a_1+a_2+\cdots+a_n$ is its $n$th partial sum with $A_0 = 0$.
Identity \eqref{BaSo} is another generalization of the KB result and contains identity \eqref{Boyadzhiev} as a special 
case for $a_n=1/n$. We also note that Frontczak's identity \eqref{Frontczak} is stated wrongly in \cite{3}. In \cite{2} Batir proved a generalization of \eqref{BaSo} involving generalized binomial coefficients. 

In this paper, we present a general framework (outlined in Section \ref{Section_GerneraFramework}) for deriving various finite combinatorial sums involving (central) binomial coefficients, harmonic numbers, odd harmonic numbers, and squared harmonic numbers. The framework can be succinctly described as a set of transformations applied to a given polynomial identity (like \eqref{Boyadzhiev} or \eqref{BaSo}) which leads to many associated combinatorial sums. For example, we will demonstrate that the finite sum 
\begin{equation}\label{eq:Intro_example}
\sum_{k = 1}^n (- 1)^{k - 1} \frac{\binom{n}{k}}{\binom{k + r}{s}} H_k = \frac{s}{r - s + 1}\left( \sum_{k = 1}^n 
\frac{1}{k\binom{n - k + r}{r - s + 1}} - \frac{H_n}{\binom{n + r}{r - s + 1}} \right),
\end{equation}
which holds for all $s,r\in\mathbb C\setminus\mathbb Z^{-}$ such that $s\ne 0$ and $r-s\not\in\mathbb Z^{-}$, is a consequence of 
\eqref{Boyadzhiev}. Identity \eqref{eq:Intro_example} contains several interesting sums as direct descendents and special cases. 
Furthermore, we will derive some presumably new identities involving binomial coefficients, such as the following identity
\begin{equation*}
\sum_{k=0}^{n} (-1)^{k} \frac{2^{2k}}{2k+1} \frac{\binom{n}{k}}{\binom{2k}{k}} = \frac{1}{2n+1}
\end{equation*}
which results from applying the general framework to Chebyshev polynomials of the second kind.

We apply the framework to discuss combinatorial sums associated with several prominent polynomial identities from the recent past. 
It is important to note that we do not explicitly formulate every possible result based on these identities to avoid repetition 
of some overly complex expression. For additional polynomial identities to which our general framework can be applied, as well as more combinatorial sums that can be derived, we refer the reader to Section~\ref{Section_Conclusion}.

\section{Definitions and required identities}

Harmonic numbers $H_s$ and odd harmonic numbers $O_s$ are defined for $0\ne s\in\mathbb C\setminus\mathbb Z^{-}$ 
by the recurrence relations
\begin{equation*}
H_s = H_{s - 1}  + \frac{1}{s} \qquad \text{and} \qquad O_s = O_{s - 1} + \frac{1}{2s - 1},
\end{equation*}
with $H_0=0$ and $O_0=0$. One way to generalize this definition is to consider harmonic numbers $H_s^{(m)}$ 
and odd harmonic numbers $O_s^{(m)}$ of order $m\in\mathbb C$ that are defined by
\begin{equation*}
H_s^{(m)} = H_{s - 1}^{(m)} + \frac{1}{s^m} \qquad \text{and} \qquad O_s^{(m)} = O_{s - 1}^{(m)} + \frac{1}{(2s - 1)^m},
\end{equation*}
with $H_0^{(m)}=0$ and $O_0^{(m)}=0$ so that $H_s=H_s^{(1)}$ and $O_s=O_s^{(1)}$. The recurrence relations imply that 
if $s=n$ is a non-negative integer, then
\begin{equation*}
H_n^{(m)} = \sum_{j = 1}^n \frac{1}{j^m} \qquad \text{and} \qquad O_n^{(m)} = \sum_{j = 1}^n \frac{1}{(2j - 1)^m}.
\end{equation*}
In particular, we have the connection between harmonic numbers and the psi (or digamma) function \cite{Srivastava}
\begin{equation}\label{gen_harmonic}
H_z = \psi (z + 1) + \gamma,
\end{equation}
where $\gamma$ is the Euler-Mascheroni constant.

\begin{lemma}\label{lem.integ}
For $u,v\in\mathbb C\setminus\mathbb Z^{-}$, we have
\begin{equation}\label{beta}
\int_0^1 x^u (1 - x)^v dx = \frac{1}{\binom{{u + v + 1}}{{u + 1}} (u + 1)}.
\end{equation}
\end{lemma}
\begin{proof}
The identity~\eqref{beta} is the well-known Beta function integral.
\end{proof}

\begin{lemma}\label{lem.ho}
If $n$ is an integer, then
\begin{gather}
H_{n - 1/2} = 2O_n - 2\ln 2 \label{eq.zts4fm1} \\
\nonumber H_{n - 1/2} - H_{- 1/2} = 2O_n,\label{eq.plh634k} \\
\nonumber H_{n - 1/2} - H_{1/2} = 2\left( {O_n - 1} \right), \label{eq.hgplrbd}\\
\nonumber H_{n + 1/2} - H_{- 1/2} = 2O_{n + 1}, \label{eq.ivi1ex5} \\
\nonumber H_{n + 1/2} - H_{1/2} = 2\left( {O_{n + 1} - 1} \right),\label{eq.u6ng5d6}\\
\nonumber H_{n + 1/2} - H_{n - 1/2} = \frac{2}{{2n + 1}},\\
\nonumber H_{n - 1/2} - H_{-3/2} = 2\left( {O_n - 1} \right)\label{eq.pobmr6h},\\
\nonumber H_{n + 1/2} - H_{-3/2} = 2\left( {O_{n + 1} - 1} \right).
\end{gather}
\end{lemma}
\begin{proof}
Use \eqref{gen_harmonic} as the definition of the harmonic numbers in conjunction with the known result for the digamma function 
at half-integer arguments \cite[Eq. (51)]{Srivastava}, namely,
\begin{equation*}
\psi (n + 1/2) = - \gamma - 2\ln 2 + 2\sum_{k = 1}^n \frac{1}{2k - 1}.
\end{equation*}
\end{proof}

\begin{lemma}\label{lem.binomial}
We have
\begin{gather}
\binom{{r + 1/2}}{s} = \frac{{\binom{{2r + 1}}{{2s}}\binom{{2s}}{s}}}{{\binom{{r}}{s}2^{2s}}}, \label{eq.binomrs1}\\
\nonumber \binom{{1/2}}{r} = \left( { - 1} \right)^{r + 1} \frac{{\binom{{2r}}{r}}}{{2^{2r} \left( {2r - 1} \right)}},\\
\binom{{r - 1/2}}{s} = \binom{{2r}}{r}\binom{{r}}{s}\frac{1}{{\binom{{2\left( {r - s} \right)}}{{r - s}}2^{2s} }},\label{eq.s83s2pw}\\
\binom{{- 1/2}}{r} = (- 1)^r \binom{{2r}}{r}2^{- 2r} \label{eq_r1div2},\\
\nonumber \binom{{r}}{{s + 1/2}} = \frac{1}{\pi }\frac{{2^{2r + 2} }}{{s + 1}}\binom{{r}}{s}\binom{{2\left( {r - s} \right)}}{{r - s}}^{ - 1} \binom{{2\left( {s + 1} \right)}}{{s + 1}}^{ - 1} .
\end{gather}
\end{lemma}
\begin{proof}
These consequences of the generalized binomial coefficients, are easy to derive using the Gamma function and some of its special values. At any rate they are 
entries in Gould's book \cite{Gould1}.
\end{proof}

\section{General results for arbitrary polynomial identities}\label{Section_GerneraFramework}

In this section, we introduce a set of general results that are applicable for a large family of polynomial identities. 
The polynomial identities under consideration are of the general form
\begin{equation}
    \sum\limits_{k=l_1}^{u_1} f(k)W_k(t) = \sum\limits_{k=l_2}^{u_2} g(k)V_k(t), \label{eq:polynomial_fgWV}
\end{equation}
where $W_k, V_k$ are polynomials belonging to the following set (we skip the dependence on $k$ and $n$ to simplify the notation)
$$BP=\large\{t^k, t^{n-k}, (1-t)^{k}, (1-t)^{n-k}, t^k(1-t)^{n-k}, t^{n-k}(1-t)^k\large\},$$
$l_1, l_2, u_1, u_2$ are non-negative integers, $f(k)$ and $g(k)$ are sequences and $t$ is a complex variable. 
These will be our standing assumptions for the rest of the paper. To deduce the results in topic, we perform a series of basic 
operations in order, which we call \textit{the (general) algorithm}. We describe that in details in the following example. 

\begin{example}\label{example:onecase}
    Suppose that the polynomial identity contains on one of the sides the following sum:
    \begin{equation}
        \sum\limits_{k=l_1}^{u_1}f(k)(1-t)^k.\label{eq:side_1}
    \end{equation}
    We now multiply that by $(1-t)^rt^{s-1}$ to obtain
    $$\sum\limits_{k=l_1}^{u_1}f(k)(1-t)^{k+r}t^{s-1}.$$
    The next step is to integrate termwise from $0$ to $1$ and use the Beta integral representation \eqref{beta}, which gives
    $$\sum\limits_{k=l_1}^{u_1}f(k)\frac{1}{s\cdot \binom{k+r+s}{s}}.$$
    The final usual step is to substitute $r\mapsto r-s$ and multiply the equation by $s$ to eliminate an unwanted factor in the denominator. This gives the final form of the identity without harmonic numbers
    \begin{equation}
        \sum\limits_{k=l_1}^{u_1}f(k)\frac{1}{\binom{k+r}{s}}\label{eq:side_1_final}
    \end{equation}
    of \eqref{eq:side_1}.
    Note that \eqref{eq:side_1} is just a part of some, unspecified, polynomial identity and a set of operations described above is performed on both sides of that identity.

We can now perform one more operation to ''generate'' harmonic numbers from \eqref{eq:side_1_final}. We differentiate that with 
respect to $s$. In case of \eqref{eq:side_1_final} this gives the following sum:
    \begin{equation}
        \label{eq:side_1_final_diff_s} 
        \sum\limits_{k=l_1}^{u_1}f(k)\frac{H_s-H_{k+r-s}}{\binom{k+r}{s}}
    \end{equation}
    where $s, r\in\mathbb{C}\setminus\mathbb{Z}^-$, $s\neq 0$ and $r-s\notin \mathbb{Z}^-$, which are again the standing assumptions for these kind of terms. We can also differentiate with respect to $r$, which gives
    \begin{equation}
        \label{eq:side_1_final_diff_r} 
        \sum\limits_{k=l_1}^{u_1}f(k)\frac{H_{k+r-s}-H_{k+r}}{\binom{k+r}{s}}.
    \end{equation}
\end{example}

The above example allows us to transform any identity \eqref{eq:polynomial_fgWV} with polynomials in the set $BP$ to an identity involving binomial coefficients and harmonic numbers. 

The polynomial identities under consideration that can be applied to the described method are of the following form (which we will refer to as \textit{standard forms})
 \allowdisplaybreaks
 \begin{align}
   \label{eq:id_gen_1}\sum\limits_{k=l_1}^{u_1}f(k)t^k&=\sum\limits_{k=l_2}^{u_2}g(k)t^k,\\
    \label{eq:id_gen_2}\sum\limits_{k=l_1}^{u_1}f(k)t^{n-k}&=\sum\limits_{k=l_2}^{u_2}g(k)t^k,\\
    \label{eq:id_gen_3}\sum\limits_{k=l_1}^{u_1}f(k)(1-t)^k&=\sum\limits_{k=l_2}^{u_2}g(k)t^k,\\
   \label{eq:id_gen_4}\sum\limits_{k=l_1}^{u_1}f(k)(1-t)^{n-k}&=\sum\limits_{k=l_2}^{u_2}g(k)t^k,\\
    \label{eq:id_gen_5}\sum\limits_{k=l_1}^{u_1}f(k)t^k(1-t)^{n-k}&=\sum\limits_{k=l_2}^{u_2}g(k)t^k,\\
    \label{eq:id_gen_6}\sum\limits_{k=l_1}^{u_1}f(k)t^{n-k}(1-t)^k&=\sum\limits_{k=l_2}^{u_2}g(k)t^k.
    \end{align}
We note that for example
$$\sum\limits_{k=l_1}^{u_1}f(k)(1+t)^k=\sum\limits_{k=l_2}^{u_2}g(k)t^k$$
is not on the above list, but substitution $t\mapsto -t$ shows it is of the form \eqref{eq:id_gen_3}. All standard forms have their right-hand-side equal and it is just a matter of simple changes to work with any possible identity in the BP set.

Let us explicitly formulate the results for identities \eqref{eq:id_gen_3} and \eqref{eq:id_gen_5}, but the very similar results 
(with obvious and necessary changes) can be derived and presented for the remaining cases.

\begin{lemma}\label{lemma:general_identity_f(k)g(k)}
Let an arbitrary polynomial identity have the following form:
 \begin{equation}\label{eq:general_idenitty_f(k)g(k)}
 \sum\limits_{k=l_1}^{u_1}f(k)(1-t)^k=\sum\limits_{k=l_2}^{u_2}g(k)t^k.
 \end{equation}
Let $r,s\in \mathbb{C}\setminus\mathbb{Z}^-$ are such that $s\neq 0$ and $r-s\notin\mathbb{Z}^-$. Then the following identity holds:
 \begin{equation}\label{eq:general_idenitty_f(k)g(k)_2}
 \sum\limits_{k=l_1}^{u_1}f(k)\frac{1}{\binom{k+r}{s}}=\sum\limits_{k=l_2}^{u_2}g(k)\frac{s}{r-s+1}\frac{1}{\binom{k+r}{r-s+1}}
 \end{equation}
Furthermore, if $s,r\in\mathbb C\setminus\mathbb Z^{-}$ such that $s\ne 0$ and $r-s\not\in\mathbb Z^{-}$, then
 \begin{align}
 \sum\limits_{k=l_1}^{u_1}f(k)\frac{H_s-H_{k+r-s}}{\binom{k+r}{s}}&=\sum\limits_{k=l_2}^{u_2}g(k)\frac{s+1}{(r-s+1)^2}\frac{1}{\binom{k+r}{r-s+1}}+\sum\limits_{k=l_2}^{u_2}g(k)\frac{s}{r-s+1}\frac{H_{k+s-1}-H_{r-s+1}}{\binom{k+r}{r-s+1}}.\label{eq:general_idenitty_f(k)g(k)_2_diff_s}
 \end{align}
\end{lemma}

In the above lemma, we differentiated with respect to $s$ to obtain \eqref{eq:general_idenitty_f(k)g(k)_2_diff_s}. In the next lemma, 
we differentiated with respect to $r$ to obtain \eqref{eq:general_idenitty_f(k)g(k)x(1-x)x_diff_s}.

\begin{lemma}\label{lemma:general_identity_f(k)g(k)x(1-x)x}
Let an arbitrary polynomial identity have the following form:
 \begin{equation}\label{eq:general_idenitty_f(k)g(k)x(1-x)x}
 \sum\limits_{k=l_1}^{u_1}f(k)t^k(1-t)^{n-k}=\sum\limits_{k=l_2}^{u_2}g(k)t^k.
 \end{equation}
Let $r,s\in \mathbb{C}\setminus\mathbb{Z}^-$ are such that $s\neq 0$ and $r-s\notin\mathbb{Z}^-$. Then the following identity holds:
 \begin{equation}\label{eq:general_idenitty_f(k)g(k)x(1-x)x}
 \sum\limits_{k=l_1}^{u_1}\frac{f(k)}{(k+s)\binom{n+r}{k+s}}=\sum\limits_{k=l_1}^{u_1}\frac{g(k)}{(k+s)\binom{k+r}{k+s}}
 \end{equation}
Furthermore, if $s,r\in\mathbb C\setminus\mathbb Z^{-}$ such that $s\ne 0$ and $r-s\not\in\mathbb Z^{-}$, then
 \begin{align}
 \sum\limits_{k=l_1}^{u_1}f(k)\frac{H_{n-k+r-s}-H_{n+r}}{(k+s)\binom{n+r}{k+s}}
 =\sum\limits_{k=l_1}^{u_1}g(k)\frac{H_{r-s}-H_{k+r}}{(k+s)\binom{k+r}{k+s}}
        \label{eq:general_idenitty_f(k)g(k)x(1-x)x_diff_s}
 \end{align}
\end{lemma}

More identities that are special cases of the above and general ones can be deduced. For example, we have the following.
\begin{corollary}\label{cor:general_identity_f(k)g(k)}
Under the assumptions of Lemma \ref{lemma:general_identity_f(k)g(k)}, we have
 \begin{align*}
 \sum\limits_{k=l_1}^{u_1}\frac{f(k)}{\binom{k+s}{s}}&=\sum\limits_{k=l_2}^{u_2}\frac{g(k)s}{k+s}, &
 \sum\limits_{k=l_1}^{u_1}\frac{f(k)}{k+1}&=\sum\limits_{k=l_2}^{u_2}\frac{g(k)}{k+1},
 \end{align*}
and 
 \begin{align*}
 \sum\limits_{k=l_1}^{u_1}f(k)\frac{H_s-H_{k}}{\binom{k+s}{s}}&=\sum\limits_{k=l_2}^{u_2}g(k)\frac{s+1}{k+s}+\sum\limits_{k=l_2}^{u_2}g(k)s   \frac{H_{k+s-1}-1}{k+s},\\
 \sum\limits_{k=l_1}^{u_1}f(k)\frac{1-H_{k}}{k+1}&=\sum\limits_{k=l_2}^{u_2}g(k)\frac{2}{k+1}+\sum\limits_{k=l_2}^{u_2}g(k)\frac{H_{k}-1}{k+1},\\
 \end{align*}
\end{corollary}
\begin{proof}
Substitute $r=s$ in Lemma \ref{lemma:general_identity_f(k)g(k)} to obtain the first identity in each set. 
Then, take $s=1$ to obtain the second one.
\end{proof}

Similar set of results can be made after applying Lemma \ref{lemma:general_identity_f(k)g(k)x(1-x)x}.
\begin{corollary}\label{cor:general_identity_f(k)g(k)x(1-x)}
Under the assumptions of Lemma \ref{lemma:general_identity_f(k)g(k)x(1-x)x}, we have
 \begin{align*}
 \sum\limits_{k=l_1}^{u_1}\frac{f(k)}{(k+s)\binom{n+s}{k+s}}&=\sum\limits_{k=l_2}^{u_2}\frac{g(k)}{k+s},&
 \sum\limits_{k=l_1}^{u_1}\frac{f(k)}{(k+1)\binom{n+1}{k+1}}&=\sum\limits_{k=l_2}^{u_2}\frac{g(k)}{k+1},
 \end{align*}
and 
 \begin{align*}
 \sum\limits_{k=l_1}^{u_1}f(k)\frac{H_{n-k}-H_{n+s}}{(k+s)\binom{n+s}{k+s}}&=-\sum\limits_{k=l_2}^{u_2}g(k)\frac{H_{k+s}}{k+s},\\
 \sum\limits_{k=l_1}^{u_1}f(k)\frac{H_{n-k}-H_{n+1}}{(k+1)\binom{n+1}{k+1}}&=-\sum\limits_{k=l_2}^{u_2}g(k)\frac{H_{k+1}}{k+1},
 \end{align*}     
\end{corollary}

The aim of the next sections is to generate identities of the above kind using some initial polynomial identity. 

\subsection{A simple application}

To give the reader an idea of how the general identities could be use and how universal they can be, let us consider the binomial identity
$$(1-t)^n=\sum\limits_{k=0}^n(-1)^{k}\binom{n}{k}t^k$$
Here, we use Corollary \ref{cor:general_identity_f(k)g(k)} with $s=1$, $l_1=u_1=n$, $l_2=0$, $u_2=n$, $f(n)=1$ and $g(k)=(-1)^k\binom{n}{k}$. This gives, after some simple algebraic manipulation, the following well-known identities:
 \begin{align}
 \label{eq:ex:_simple_binom_app_1}\frac{1}{n+1}&=\sum\limits_{k=0}^n(-1)^k\binom{n}{k}\frac{1}{k+1},\\
 \frac{H_n}{n+1}&=\sum\limits_{k=0}^n(-1)^{k+1}\binom{n}{k}\frac{H_{k+1}}{k+1}.
 \end{align}
To see the application of Corollary \ref{cor:general_identity_f(k)g(k)x(1-x)}, let us consider
$$\sum\limits_{k=0}^nt^k(1+t)^{n-k}=(t+1)^{n+1}-t^{n+1}=\sum\limits_{k=0}^n\binom{n+1}{k}t^k.$$
From that we get:
\begin{align*}
 \sum\limits_{k=0}^n\frac{(-1)^k}{(k+1)\binom{n+1}{k+1}}&=\sum\limits_{k=0}^n\frac{(-1)^k}{k+1}\binom{n+1}{k},\\
 \sum\limits_{k=0}^n(-1)^k\frac{H_{n-k}-H_{n+1}}{(k+1)\binom{n+1}{k+1}}&=\sum\limits_{k=0}^n(-1)^{k+1}\binom{n+1}{k}\frac{H_{k+1}}{k+1}.
\end{align*}
The first identity may look curious, but in fact both sides of it are $0$, which follows directly from \eqref{eq:ex:_simple_binom_app_1}.

\section{Applications}

\subsection{Sums associated with polynomial identities of Knuth-Boyadzhiev and Frontczak}

Recall that the identity of Knuth and Boyadzhiev \eqref{Boyadzhiev} can be stated as 
\begin{equation}
\sum_{k = 1}^n (- 1)^{k - 1} \binom{n}{k} H_k t^k = - (1 - t)^n H_n + \sum_{k = 1}^n \frac{(1 - t)^{n - k}}{k}, \label{eq.1}\tag{A}
\end{equation}
whereas Frontczak's compliment \eqref{Frontczak} has the form
\begin{equation}
\sum_{k=0}^{n} (-1)^k \binom{n}{k} H_{k} t^k = \left ((1-t)^n - 1 \right )H_n + t \sum_{k=0}^{n-1} H_{n-1-k} (1-t)^k. \label{eq.2}\tag{B}
\end{equation}

These forms are members of the polynomial family introduced in the previous section. The next results are therefore immediate. 

\begin{theorem}\label{boya.thm1}
If $n$ is a non-negative integer, then
\begin{equation}\label{eq.noy1xtq}
\sum_{k = 0}^n (- 1)^k \binom{{n}}{k} \frac{(2k + 1)}{2^{2k}}\binom{{2k}}{k} H_{k + 1} = 
\frac{\binom{{2n}}{n}}{2^{2n}} \frac{H_{n + 1}}{n + 1} - \frac{2}{n + 1}\sum_{k = 1}^n \frac{\binom{2(n - k)}{n - k}}{2^{2(n - k)} k},
\end{equation}
and, more generally, if $s,r\in\mathbb C\setminus\mathbb Z^{-}$ such that $s\ne 0$ and $r-s\not\in\mathbb Z^{-}$, then
\begin{equation}\label{eq.bhi8kzc}
\sum_{k = 1}^n (- 1)^{k - 1} \frac{\binom{n}{k}}{\binom{k + r}{s}} H_k = \frac{s}{r - s + 1}\left( \sum_{k = 1}^n 
\frac{1}{k\binom{n - k + r}{r - s + 1}} - \frac{H_n}{\binom{n + r}{r - s + 1}} \right).
\end{equation}
\end{theorem}
\begin{proof}

To obtain~\eqref{eq.bhi8kzc}, we apply the general algorithm to \eqref{eq.1}. Identity~\eqref{eq.noy1xtq} is obtained 
by setting $r=-1$, $s=-1/2$ in~\eqref{eq.bhi8kzc}. Note that we used
\begin{equation*}
\lim_{x\to 0} \frac{1}{\binom{{x - 1}}{{1/2}}} = 0.
\end{equation*}
\end{proof}

\begin{corollary}\label{Frisch_thm7}
For any $r\in\mathbb C\setminus\mathbb Z^{-}$ we have
\begin{equation}
\sum_{k = 1}^n (- 1)^{k - 1} \frac{\binom{n}{k}}{\binom{k + r}{r}} H_k = \frac{r}{n+r}\left( H_{n+r} - H_{r} \right) + \frac{n}{(n+r)^2}.
\end{equation}
\end{corollary}
\begin{proof}
Set $s=r$ in \eqref{eq.bhi8kzc}. This gives
\begin{equation*}
\sum_{k = 1}^n (- 1)^{k - 1} \frac{\binom{n}{k}}{\binom{k + r}{r}} H_k = r\sum_{k=1}^n \frac{1}{k(n-k+r)} - \frac{r}{n+r} H_n.
\end{equation*}
But
\begin{align*}
\sum_{k=1}^n \frac{1}{k(n-k+r)} &= \frac{1}{n+r}\sum_{k=1}^n \left (\frac{1}{k} + \frac{1}{n-k+r} \right ) \\
&= \frac{1}{n+r} \left ( H_n + H_{n-1+r} - H_{r-1} \right ) 
\end{align*}
and hence
\begin{equation*}
\sum_{k = 1}^n (- 1)^{k - 1} \frac{\binom{n}{k}}{\binom{k + r}{r}} H_k = \frac{r}{n+r}\left( H_{n-1+r} - H_{r-1} \right).
\end{equation*}
\end{proof}

\begin{remark}
Corollary \ref{Frisch_thm7} is Corollary 1.2 in Chu's paper \cite{Chu12} and also Theorem 7 in \cite{Adegoke2}.
\end{remark}

\begin{theorem}
If $n$ is a non-negative integer, then
\begin{equation}
\begin{split}
\sum_{k = 0}^n {\left( { - 1} \right)^k \frac{{2k + 1}}{{2^{2k} \left( {k + 1} \right)}}\binom{{n}}{k}\binom{{2k}}{k}H_{k + 1} O_{k + 1} }  &=  - \frac{{\binom{{2n}}{n}}}{{2^{2n} \left( {n + 1} \right)}}H_{n + 1} \left( {O_n  - 1} \right)\\
&\qquad + \frac{1}{{n + 1}}\sum_{k = 1}^n {\frac{{\binom{{2\left( {n - k} \right)}}{{n - k}}}}{{2^{2\left( {n - k} \right)} k}}\left( {O_{n - k}  - 1} \right)} ,
\end{split}
\end{equation}
and, more generally, if $s,r\in\mathbb C\setminus\mathbb Z^{-}$ such that $s\ne 0$ and $r-s\not\in\mathbb Z^{-}$, then
\begin{equation}\label{eq.boyhar2}
\begin{split}
&  \sum_{k = 1}^n \left( { - 1} \right)^{k - 1} \binom{{n}}{k}H_k \frac{{H_{k + r - s} - H_s }}{{\binom{{k + r}}{s}}} \\
&\qquad = -\frac{{r + 1}}{{\left( {r - s + 1} \right)^2 }}\sum_{k = 1}^n {\frac{1}{{k\binom{{n - k + r}}{{r - s + 1}}}}} - \frac{s}{{r - s + 1}}\sum_{k = 1}^n \frac{{H_{n - k + s - 1} - H_{r - s + 1} }}{{k\binom{{n - k + r}}{{r - s + 1}}}} \\
&\qquad\qquad + \frac{{ \left( {r + 1} \right)H_n}}{{\left( {r - s + 1} \right)^2 \binom{{n + r}}{{r - s + 1}}}} - \frac{{sH_n \left( {H_{r - s + 1} - H_{n + s - 1} } \right)}}{{\left( {r - s + 1} \right)\binom{{n + r}}{{r - s + 1}}}}.
\end{split}
\end{equation}
\end{theorem}
\begin{proof}
Differentiate~\eqref{eq.bhi8kzc} with respect to $s$.
\end{proof}

\begin{corollary}
For any nonzero $r\in\mathbb C\setminus\mathbb Z^{-}$ we have
\begin{align}
\sum_{k = 1}^n (- 1)^{k - 1} \frac{\binom{n}{k}}{\binom{k + r}{r}} H^2_k &= \frac{1}{n+r} (H_{n+r}-H_r)(r H_r - 1) 
+ \frac{n}{(n+r)^2}\left (H_r - \frac{1}{r}\right ) \nonumber \\
& \qquad + \frac{r}{n+r} H_n H_{n+r-1} - r \sum_{k=1}^n \frac{H_{n-k+r-1}}{k(n-k+r)}.
\end{align}
In particular,
\begin{equation}\label{Choi_id}
\sum_{k = 1}^n (- 1)^{k - 1} \frac{\binom{n}{k}}{k+1} H^2_k = \frac{1}{2(n+1)}\left ( 3H_n^{(2)} - H_n^2 \right ).
\end{equation}
\end{corollary}
\begin{proof}
Set $s=r$ in \eqref{eq.boyhar2} and simplify using Corollary \ref{Frisch_thm7}. The special case follows upon setting $r=1$ in conjunction with
\begin{equation*}
\sum_{k=1}^n \frac{H_{n-k}}{k} = H_n^2 - H_n^{(2)},
\end{equation*}
and
\begin{equation*}
\sum_{k=1}^n \frac{H_{n-k}}{n-k+1} = \frac{1}{2} \left ( H_n^2 - H_n^{(2)}\right ).
\end{equation*}
\end{proof}

\begin{remark}
Identity \eqref{Choi_id} is Equation (2.22) in Theorem 2 of Choi \cite{Choi13}. There is also the identity (corresponding to $r=0$ in the above relation)
\begin{equation}
\sum_{k=1}^{n} (-1)^k \binom{n}{k} H_k^2 = \frac{H_n}{n} - \frac{2}{n^2}.
\end{equation}
This identity has been obtained by Wang \cite{19} by the method of Riordan arrays, and rediscovered by Boyadzhiev \cite{Boya2014}.
\end{remark}

A different expression for the sum \eqref{eq.bhi8kzc} can be derived from Frontczak's identity \eqref{eq.2}.

\begin{theorem}\label{thm.fr1}
If $s,r\in\mathbb C\setminus\mathbb Z^{-}$ such that $s\ne 0$ and $r-s\not\in\mathbb Z^{-}$, then
\begin{equation}\label{eq.thmfr1}
\sum_{k=0}^n (- 1)^{k} \frac{\binom{n}{k}}{\binom{k + r}{s}} H_k = H_n \left (\frac{s}{r - s + 1}\frac{1}{\binom{n + r}{r - s + 1}} 
- \frac{1}{\binom{r}{s}}\right ) + s \sum_{k=0}^{n-1}\frac{H_{n-1-k}}{(k+s)\binom{r+k+1}{r-s+1}}.
\end{equation}
\end{theorem}
\begin{proof}
Apply the general algorithm to \eqref{eq.2} and simplify.
\end{proof}

\begin{corollary}
If $s,r\in\mathbb C\setminus\mathbb Z^{-}$ such that $s\ne 0$ and $r-s\not\in\mathbb Z^{-}$, then
\begin{equation}
\sum_{k=0}^{n-1}\frac{H_{n-1-k}}{(k+s)\binom{r+k+1}{r-s+1}} = \frac{H_n}{s \binom{r}{s}} - 
\frac{1}{r - s + 1}\sum_{k=0}^n \frac{1}{k \binom{n-k+r}{r-s+1}}.
\end{equation}
In particular, we have for all $s\geq 1$
\begin{equation}
\sum_{k=0}^{n}\frac{H_{n-k}}{(k+s)(k+1+s)} = \frac{n+1}{s(n+1+s)}H_{n+1} - \frac{1}{n+1+s}(H_{n+s} - H_{s-1})
\end{equation}
and
\begin{align}\label{eq.cor001}
\sum_{k=0}^{n}\frac{H_{n-k}}{(k+s)(k+1+s)(k+2+s)} &= \frac{1}{2} \Big (\frac{H_{n+1}}{s(s+1)} - \frac{1}{n+1+s}(H_{n+1} + H_{n+s} - H_{s-1}) \nonumber \\
& \qquad + \frac{1}{n+2+s}(H_{n+1} + H_{n+1+s} - H_{s}) \Big ). 
\end{align}
\end{corollary}
\begin{proof}
Compare \eqref{eq.thmfr1} with \eqref{eq.bhi8kzc}. The first particular case follows by setting $r=s$, replacing $n$ by $n+1$, 
and using the same evaluation for the sum as in the proof of Corollary \ref{Frisch_thm7}. To get the sum \eqref{eq.cor001} set $r=s+1$, replace $n$ by $n+1$, and make use of the partial fraction decomposition
\begin{equation*}
\frac{1}{k(n-k+s)(n-k+s+1)} = \frac{1}{k(n-k+s)} - \frac{1}{k(n-k+s+1)}.
\end{equation*}
\end{proof}

\begin{theorem}
If $s,r\in\mathbb{C}\setminus \mathbb{Z}$ are such that $s\neq 0$ and $r-s\notin\mathbb{Z}$, then
\begin{align*}
\sum_{k=0}^n (- 1)^{k}\binom{n}{k} \frac{H_s-H_{k+r-s}}{\binom{k + r}{s}} H_k 
& = H_n\left(\frac{r+1}{(r-s+1)^2}\frac{1}{\binom{n + r}{r - s + 1}}+\frac{s}{r-s+1}\frac{H_{n+s-1}-H_{r-s+1}}{\binom{n+r}{r-s+1}}\right) \\
&\quad + H_n \frac{H_{r-s}-H_s}{\binom{r}{s}} \\
&\quad + \sum_{k=0}^{n-1} H_{n-1-k}\frac{s(k+s)(H_{k+s}-H_{r-s+1})+k}{(k+s)^2\binom{n+r}{r-s+1}}.
\end{align*}
\end{theorem}
\begin{proof}
Differentiate \eqref{eq.thmfr1} with respect to $s$.
\end{proof}

\begin{corollary}
We have
\begin{align*}
\sum_{k=0}^n (-1)^k \binom{n}{k}\frac{H_k^2}{\binom{k+r}{r}}
=& H_n H_r-H_n\left(\frac{1}{n+r}+\frac{r}{n+r}H_{n+r-1}\right) - \frac{r}{n+r}(H_{n-1+r}-H_{r-1})\\
& -\frac{r}{n+r}\sum_{k=0}^{n-1} \frac{H_{n-1-k}H_{k+r}}{k+r} + \frac{r}{n+r}\sum_{k=0}^{n-1}\frac{H_{n-k-1}}{k+r}\\
& -\frac{1}{n+r}\sum_{k=0}^{n-1} \frac{k H_{n+k}}{(k+r)^2}
\end{align*}
and in particular,
\begin{align*}
\sum_{k=0}^n (-1)^k \binom{n}{k} \frac{H_k^2}{k+1} &= \frac{n H_n-H_n-H^2}{n+1} + \frac{1}{n+1}(H_n^2+H_n^{(2)})\\
&\qquad -\frac{1}{n+1}\sum_{k=1}^n \frac{H_{n-k}H_k}{k} - \frac{1}{n+1}\sum_{k=0}^{n-1} \frac{k H_{n+k}}{(k+1)^2}.
\end{align*}
\end{corollary}
\begin{proof}
Set $s=r$ in ... to obtain
\begin{align*}
\sum_{k=0}^n (- 1)^{k}\binom{n}{k} \frac{H_r-H_{k}}{\binom{k + r}{r}} H_k&=H_n\left(\frac{r+1}{n+r}+\frac{r}{n+r}(H_{n+r-1}-1)\right)\\
&\quad - H_n H_r \\
&\quad + \sum_{k=0}^{n-1} H_{n-1-k}\frac{r(k+r)(H_{k+r}-1)+k}{(k+r)^2(n+r)}.
\end{align*}
Use Corollary \ref{Frisch_thm7} to the left-hand side and rearrange. This gives the main identity. Particular case follows upon setting $r=1$.
\end{proof}

\subsection{Sums derived from polynomial identities involving $H_n^2$ and $H_n^{(2)}$} 

Next, we recall the following two identities.

\begin{theorem}
Let $n\in\mathbb{N}$ and $t\in\mathbb{C}$. Then we have
\begin{equation}
\sum_{k=0}^{n} \binom{n}{k} H_k^2\,t^k = H_n^2 (1+t)^{n} - \sum_{k=1}^{n} \frac{H_n - 2H_{k-1} + H_{n-k}}{k} (1+t)^{n-k} \label{eq.3}\tag{C}
\end{equation}
and
\begin{equation}
\sum_{k=0}^{n} \binom{n}{k} H_k^{(2)}\, t^k = H_n^{(2)} (1+t)^{n} - \sum_{k=1}^{n} \frac{H_n - H_{n-k}}{k} (1+t)^{n-k} \label{eq.4}\tag{D}. 
\end{equation}
\end{theorem}
\begin{proof}
Identity \eqref{eq.3} is a restatement of Sofo's and Batir's result from \cite{3}
\begin{align*}
\sum_{k=0}^{n} \binom{n}{k} H_k^2 t^k = (1+t)^{n} \bigg( H_n^2 - \sum_{k=1}^{n} \frac{H_n - 2H_k + H_{n-k}}{k(1+t)^{k}} 
- 2\sum_{k=1}^{n} \frac{1}{k^2(1+t)^{k}}\bigg),
\end{align*}
while \eqref{eq.4} follows from combining this identity with 
\begin{equation*}
\sum_{k=0}^{n} \binom{n}{k} \left (H_k^2 - H_k^{(2)}\right ) t^k = \left (H_n^2 - H_n^{(2)}\right ) (1+t)^{n} 
+ 2 \sum_{k=1}^{n} \frac{H_{k-1} - H_{n-k}}{k} (1+t)^{n-k}, 
\end{equation*}
which comes from \cite{Adegoke}. Identity (D) is also stated as Identity 24 in Choi's paper \cite{6}.
\end{proof}

Using the general algorithm (as in Theorems \ref{boya.thm1} and \ref{thm.fr1}) we can deduce the following results.

\begin{theorem}\label{thm.harorder2}
If $s,r\in\mathbb C\setminus\mathbb Z^{-}$ such that $s\ne 0$ and $r-s\not\in\mathbb Z^{-}$, then
\begin{equation}\label{eq.harorder2}
\sum_{k=0}^n (- 1)^{k} \frac{\binom{n}{k}}{\binom{k + r}{s}} H_k^{(2)} = \frac{s}{r - s + 1}\left ( \frac{H_n^{(2)}}{\binom{n + r}{r - s + 1}} 
- \sum_{k=1}^{n}\frac{H_n - H_{n-k}}{k \binom{n-k+r}{r-s+1}} \right ).
\end{equation}
\end{theorem}

\begin{corollary}
For all $r\geq 1$ we have
\begin{equation}
\sum_{k=0}^n (- 1)^{k+1} \frac{\binom{n}{k}}{\binom{k + r}{r}} H_k^{(2)} = \frac{r}{n+r}\left ( H_n (H_{n-1+r}-H_{r-1}) 
- \sum_{k=0}^{n-1}\frac{H_{k}}{k+r} \right ).
\end{equation}
In particular,
\begin{equation}
\sum_{k=0}^n (- 1)^{k+1} \binom{n}{k} \frac{H_k^{(2)}}{k+1} = \frac{1}{2(n+1)}\left ( H_n^2 + H_n^{(2)} \right )
\end{equation}
and
\begin{equation}
\sum_{k=0}^n (- 1)^{k} \binom{n}{k} \frac{H_k^{(2)}}{(k+1)(k+2)} = \frac{1}{n+1}\left ( \frac{1}{2} \left (H_{n+1}^2 - H_{n+1}^{(2)} \right )
+ H_n - H_n H_{n+1} - \frac{n}{n+1} \right ).
\end{equation}
\end{corollary}

\begin{remark}
Choi \cite{Choi13} states in Theorem 2, Equation (2.23), the result
\begin{equation*}
\sum_{k=1}^n (- 1)^{k+1} \binom{n}{k} \frac{H_k^{(2)}}{k+1} = \frac{1}{2(n+1)}\left ( 5 H_n^{(2)} - 3 H_n^2\right ),
\end{equation*}
which is not correct.
\end{remark}

\subsection{Sums associated with an identity of Dattoli et al.}

In this subsection we work with the following polynomial identity due to Datolli et al. \cite{Dattoli}: 
\begin{equation}\label{eq.5}\tag{E}
\sum_{k=0}^n \binom {n}{k} (-1)^k \frac{1}{k+2} t^{k} (1+t)^{n-k} = \sum_{k=0}^n \binom {n}{k} \frac{t^{k}}{(k+1)(k+2)}.
\end{equation}

Using the general algorithm again (as in previous sections) we can transform \eqref{eq.5} into our next result.

\begin{theorem}\label{final_1}
If $s,r\in\mathbb C\setminus\mathbb Z^{-}$ such that $s\ne 0$ and $r-s\not\in\mathbb Z^{-}$, then
\begin{equation}\label{eq1.final}
\sum_{k=0}^n \binom{n}{k} \frac{1}{(k+2)(k+s)\binom{n + r}{k+s}} = \sum_{k=0}^{n}\binom{n}{k} \frac{(-1)^k}{(k+1)(k+2)(k+s)\binom{k + r}{k+s}}.
\end{equation}
\end{theorem}

\begin{corollary}
For all $r\geq 1$ we have
\begin{equation}\label{eq2.final}
\sum_{k=0}^{n}\binom{n}{k} \frac{(-1)^k}{(k+1)(k+2)(k+r)} = \frac{1}{(n+1)\cdots (n+r)}\sum_{k=0}^n \frac{\prod_{j=1}^{r-1} (k+j)}{k+2},
\end{equation}
where the product equals 1 for $r=1$. In particular, we have the identities
\begin{equation}
\sum_{k=0}^{n} \binom{n}{k} \frac{(-1)^k}{(k+1)^2 (k+2)} = \frac{H_{n+2}-1}{n+1},
\end{equation}
\begin{equation}
\sum_{k=0}^{n} \binom{n}{k} \frac{(-1)^k}{(k+1) (k+2)^2} = \frac{n+2-H_{n+2}}{(n+1)(n+2)},
\end{equation}
and
\begin{equation}
\sum_{k=0}^{n} \binom{n}{k} \frac{(-1)^k}{(k+1)(k+2)(k+3)} = \frac{1}{2(n+3)}.
\end{equation}
\end{corollary}
\begin{proof}
Set $r=s$ in Theorem \ref{final_1} and simplify using
\begin{equation*}
\binom{n+r}{k+r} = \frac{(n+r)\cdots (n+1)}{(k+r)\cdots (k+1)} \binom{n}{k}.
\end{equation*}
The particular examples correspond to $r=1$, $r=2$ and $r=3$, respectively, where it was used that
\begin{equation*}
\sum_{k=0}^n \frac{1}{k+2} = H_{n+2} - 1,
\end{equation*}
\begin{equation*}
\sum_{k=0}^n \frac{k+1}{k+2} = n + 2 - H_{n+2},
\end{equation*}
and
\begin{equation*}
\sum_{k=0}^n (k+1) = \frac{(n+1)(n+2)}{2}.
\end{equation*}
\end{proof}

\begin{theorem}\label{final_2}
If $s,r\in\mathbb C\setminus\mathbb Z^{-}$ such that $s\ne 0$ and $r-s\not\in\mathbb Z^{-}$, then
\begin{align}\label{eq3.final}
& \sum_{k=0}^n \binom{n}{k} \frac{H_{n-k+r-s}}{(k+2)(k+s)\binom{n + r}{k+s}} - H_{n+r} \sum_{k=0}^n \binom{n}{k} \frac{1}{(k+2)(k+s)\binom{n + r}{k+s}} \nonumber \\
&\quad = H_{r-s} \sum_{k=0}^n \binom{n}{k} \frac{(-1)^k}{(k+1)(k+2)(k+s)\binom{k + r}{k+s}} - \sum_{k=0}^{n}\binom{n}{k} \frac{(-1)^k H_{k+r}}{(k+1)(k+2)(k+s)\binom{k + r}{k+s}}.
\end{align}
\end{theorem}
\begin{proof}
Differentiate \eqref{eq1.final} with respect to $r$.
\end{proof}

\begin{corollary}
For all $r\geq 1$ we have
\begin{align}\label{eq4.final}
\sum_{k=0}^{n}\binom{n}{k} \frac{(-1)^k H_{k+r}}{(k+1)(k+2)(k+r)} &= 
\frac{1}{(n+1)\cdots (n+r)} \Big ( H_{n+r} \sum_{k=0}^n \frac{\prod_{j=1}^{r-1} (k+j)}{k+2} \nonumber \\
&\qquad - \sum_{k=0}^n \frac{H_{n-k}}{k+2} \prod_{j=1}^{r-1} (k+j) \Big ).
\end{align}
In particular, we have the identities
\begin{equation}
\sum_{k=0}^{n} \binom{n}{k} \frac{(-1)^k H_{k+1}}{(k+1)^2 (k+2)} 
= \frac{1}{n+1}\left ( H_{n+1}(H_{n+2}-1) - \sum_{k=0}^n \frac{H_{n-k}}{k+2}\right ),
\end{equation}
\begin{equation}
\sum_{k=0}^{n} \binom{n}{k} \frac{(-1)^k H_{k+2}}{(k+1) (k+2)^2} = \frac{1}{n+1} 
- \frac{1}{(n+1)(n+2)}\left ( H_{n+2}^2 - H_{n+1} - \sum_{k=0}^n \frac{H_{n-k}}{k+2} \right ),
\end{equation}
and
\begin{equation}
\sum_{k=0}^{n} \binom{n}{k} \frac{(-1)^k H_{k+3}}{(k+1)(k+2)(k+3)} = \frac{1}{2(n+3)}\left ( H_{n+3} - H_{n+1} + \frac{3n+4}{2(n+2)}\right ).
\end{equation}
\end{corollary}
\begin{proof}
Set $r=s$ in Theorem \ref{final_2} and simplify. The particular examples correspond to $r=1$, $r=2$ and $r=3$, respectively. 
Note also that we used 
\begin{equation*}
\sum_{k=0}^n H_{n-k} = \sum_{k=0}^n H_k = (n+1)(H_{n+1}-1)
\end{equation*}
as well as
\begin{equation*}
\sum_{k=0}^n k H_k = \frac{1}{4} n (n+1)(2H_{n+1}-1).
\end{equation*}
\end{proof}

\begin{theorem}\label{final_3}
If $s,r\in\mathbb C\setminus\mathbb Z^{-}$ such that $s\ne 0$ and $r-s\not\in\mathbb Z^{-}$, then
\begin{align}\label{eq5.final}
& \sum_{k=0}^n \binom{n}{k} \frac{H_{k+s}-H_{n-k+r-s}}{(k+2)(k+s)\binom{n + r}{k+s}} - \sum_{k=0}^n \binom{n}{k} \frac{1}{(k+2)(k+s)^2
\binom{n + r}{k+s}} \nonumber \\
&\quad = \sum_{k=0}^n \binom{n}{k} \frac{(-1)^k (H_{k+s}-H_{r-s})}{(k+1)(k+2)(k+s)\binom{k + r}{k+s}} - \sum_{k=0}^{n}\binom{n}{k} \frac{(-1)^k}{(k+1)(k+2)(k+s)^2\binom{k + r}{k+s}}.
\end{align}
\end{theorem}
\begin{proof}
Differentiate \eqref{eq1.final} with respect to $s$.
\end{proof}

\begin{corollary}
For all $r\geq 1$ we have
\begin{equation}\label{eq6.final}
\sum_{k=0}^{n}\binom{n}{k} \frac{(-1)^k}{(k+1)(k+2)(k+r)^2} = \frac{1}{(n+1)\cdots (n+r)} \sum_{k=0}^n \Big ( H_{n+r} - H_{k+r} + \frac{1}{k+r} \Big ) \frac{\prod_{j=1}^{r-1} (k+j)}{k+2}.
\end{equation}
In particular, we have the identities
\begin{equation}
\sum_{k=0}^{n} \binom{n}{k} \frac{(-1)^k}{(k+1)^3 (k+2)} 
= \frac{1}{n+2} + \frac{H_{n+1}(H_{n+2}-1)}{n+1} - \frac{H_{n+2}^2 - H_{n+2}^{(2)}}{2(n+1)},
\end{equation}
\begin{equation}
\sum_{k=0}^{n} \binom{n}{k} \frac{(-1)^k}{(k+1) (k+2)^3} = \frac{1}{n+1} - \frac{H_{n+2}^2 + H_{n+2}^{(2)}}{2(n+1)(n+2)},
\end{equation}
and
\begin{equation}
\sum_{k=0}^{n} \binom{n}{k} \frac{(-1)^k}{(k+1)(k+2)(k+3)^2} = \frac{H_{n+3}}{(n+1)(n+2)(n+3)} + \frac{n-2}{4(n+1)(n+2)}.
\end{equation}
\end{corollary}
\begin{proof}
Set $r=s$ in Theorem \ref{final_3}. This gives
\begin{align*}
&\frac{1}{(n+1)\cdots (n+r)} \Big (\sum_{k=0}^n (H_{k+r}-H_{n-k}) \frac{(k+1)\cdots (k+r-1)}{k+2} - \sum_{k=0}^n \frac{(k+1)\cdots (k+r-1)}{k+2}
\Big ) \\
&\quad = \sum_{k=0}^n \binom{n}{k} \frac{(-1)^k H_{k+r}}{(k+1)(k+2)(k+r)} - \sum_{k=0}^{n}\binom{n}{k} \frac{(-1)^k}{(k+1)(k+2)(k+r)^2}.
\end{align*}
Now, simplify using \eqref{eq4.final}. The particular examples correspond to $r=1$, $r=2$ and $r=3$, respectively, after some lengthy calculations. Here we used 
\begin{equation*}
\sum_{k=0}^n \frac{H_{k+1}}{k+2} = \frac{1}{2}\left (H_{n+2}^2 - H_{n+2}^{(2)}\right ),
\end{equation*}
\begin{equation*}
\sum_{k=0}^n \frac{1}{(k+1)(k+2)} = \frac{n+1}{n+2},
\end{equation*}
\begin{equation*}
\sum_{k=0}^n \frac{H_{k+2}}{k+2} = \frac{1}{2}\left (H_{n+2}^2 + H_{n+2}^{(2)}\right ) - 1,
\end{equation*}
and
\begin{equation*}
\sum_{k=0}^n H_{k+3} (k+1) = \frac{1}{4}\left (2(n-1)(n+4)H_{n+4} - n^2 + n + 24 \right ),
\end{equation*}
all being standard sums.
\end{proof}

\subsection{Sums associated with an identity for the Chebyshev polynomials of the second kind}

This application deals with sums that are derived using an identity for the Chebyshev polynomials of the second kind $U_n(t)$
defined by 
\begin{equation*}
U_0(t) = 1,\quad U_1(t) = 2t, \quad U_{n+1}(t) = 2t U_n(t) - U_{n-1}(t).
\end{equation*}
From \cite[Lemma 41]{Adegoke2023} we have that
\begin{equation*}
\sum_{k=0}^n (-1)^k 2^{2k} \binom{n+k}{n-k} t^{2k} = (-1)^n U_{2n}(t).
\end{equation*}
Using the representation    
\begin{equation*}
U_n(t) = \sum_{k=0}^{\lfloor{n}/{2}\rfloor} \binom{n+1}{2k+1} (t^2-1)^k t^{n-2k}, 
\end{equation*}
we immediately get the identity
\begin{equation}\label{eq.Cheb}\tag{F}
\sum_{k=0}^n (-1)^k 2^{2k} \binom{n+k}{n-k} t^{k} = \sum_{k=0}^n (-1)^{n-k} \binom{2n+1}{2k+1} (1-t)^k t^{n-k}
\end{equation}
and the general algorithm can be applied to transform \eqref{eq.Cheb} into our next results.

\begin{theorem}\label{Cheb_1}
If $s,r\in\mathbb C\setminus\mathbb Z^{-}$ such that $s\ne 0$ and $r-s\not\in\mathbb Z^{-}$, then
\begin{equation}\label{Cheb.eq1}
\sum_{k=0}^n (-1)^k 2^{2k} \frac{\binom{n+k}{n-k}}{(k+s) \binom{k+r}{k+s}} 
= \sum_{k=0}^{n} (-1)^{n-k} \frac{\binom{2n+1}{2k+1}}{(n-k+s)\binom{n+r}{n-k+s}}.
\end{equation}
\end{theorem}

\begin{corollary}
For all $r\geq 1$ we have
\begin{equation}\label{Cheb.eq2}
\sum_{k=0}^n (-1)^k 2^{2k}\frac{\binom{n+k}{n-k}}{k+r} = \sum_{k=0}^{n} (-1)^{n-k} \frac{\binom{2n+1}{2k+1}}{(n-k+r)\binom{n+r}{k}}.
\end{equation}
In particular, we have the identity
\begin{equation}\label{Cheb_binfrac}
\sum_{k=0}^{n} (-1)^{n-k} \frac{\binom{2n+1}{2k+1}}{\binom{n}{k}} = \frac{(-1)^n (2n+1)-1}{2n} \qquad (n\geq 1).
\end{equation}
\end{corollary}
\begin{proof}
Set $r=s$ in Theorem \ref{Cheb_1} and simplify. The particular example corresponds to $r=1$ in conjunction to
\begin{equation*}
\sum_{k=0}^n (-1)^k 2^{2k}\frac{\binom{n+k}{n-k}}{k+1} = \frac{(-1)^n (2n+1)-1}{2n(n+1)},
\end{equation*}
where the last expression is true since
\begin{equation*}
\sum_{k=0}^n (-1)^{n-k} 2^{2k}\frac{\binom{n+k}{n-k}}{k+1} = \int_0^1 U_{2n}(\sqrt{t})\,dt = \frac{2n - \cos(\pi n) + 1}{2n(n+1)}.
\end{equation*}
\end{proof}

\begin{corollary}
We have
\begin{equation}\label{eqincor:sum=1/2n+1}
\sum_{k=0}^{n} (-1)^{k} \frac{2^{2k}}{2k+1} \frac{\binom{n}{k}}{\binom{2k}{k}} = \frac{1}{2n+1}
\end{equation}
and
\begin{equation}
\sum_{k=0}^{n} (-1)^{k+1} 2^{2k} \frac{k}{2k+1} \frac{\binom{n}{k}}{\binom{2k}{k}} = \frac{2n}{(2n+1)(2n-1)}.
\end{equation}
\end{corollary}
\begin{proof}
Set $r=s=1/2$ in Theorem \ref{Cheb_1}. Using \eqref{eq.binomrs1} this choice gives
\begin{equation}\label{eq:cor28sum}
\sum_{k=0}^n (-1)^k 2^{2k}\frac{\binom{n+k}{n-k}}{2k+1} 
= \sum_{k=0}^{n} (-1)^{n-k} \frac{2^{2k}}{2(n-k)+1} \frac{\binom{n}{k} \binom{2n+1}{2k+1}}{\binom{2k}{k} \binom{2n+1}{2k}}.
\end{equation}
The binomial coefficients can be simplified to
\begin{equation*}
\frac{\binom{2n+1}{2k+1}}{\binom{2n+1}{2k}} = \frac{2(n-k)+1}{2k+1}.
\end{equation*}
The final step in the proof is to use
\begin{equation*}\label{eq:sum-1^n/2n+1}
\sum_{k=0}^n (-1)^k 2^{2k}\frac{\binom{n+k}{n-k}}{2k+1} = \frac{(-1)^n}{2n+1},
\end{equation*}
which follows from
\begin{equation*}
\sum_{k=0}^n (-1)^{n-k} 2^{2k}\frac{\binom{n+k}{n-k}}{2k+1} = \int_0^1 U_{2n}(t)\,dt = \frac{\sin(\pi n) + 1}{2n+1}.
\end{equation*}
This completes the proof of the first identity. The second follows as a combination with
\begin{equation*}\label{eq:sum=1/2n-1}
\sum_{k=0}^{n} (-1)^{k+1} 2^{2k} \frac{\binom{n}{k}}{\binom{2k}{k}} = \frac{1}{2n-1}.
\end{equation*}
\end{proof}

\begin{corollary}
We have
\begin{equation}
\sum_{k=0}^{n} (-1)^{n-k} \frac{2^{2k}}{2(n-k)+3} \frac{\binom{n+1}{k} \binom{2n+1}{2k+1}}{\binom{2k}{k} \binom{2n+3}{2k}} 
= \frac{(-1)^n (4n^2+4n-1)}{(2n-1)(2n+1)(2n+3)}.
\end{equation}
\end{corollary}
\begin{proof}
Set $r=s=3/2$ in Theorem \ref{Cheb_1} and use \eqref{eq.binomrs1} again. The left hand side can be evaluated as
\begin{equation*}
\sum_{k=0}^n (-1)^{n-k} 2^{2k}\frac{\binom{n+k}{n-k}}{2k+3} = \int_0^1 t^2 U_{2n}(t)\,dt = \frac{4n^2+4n-1}{(2n-1)(2n+1)(2n+3)}.
\end{equation*}
\end{proof}

\begin{corollary}
We have
\begin{equation}
\sum_{k=0}^{n} (-1)^{n-k} k \frac{\binom{2n+1}{2k+1}}{\binom{n}{k}} 
= \begin{cases}
 1, & \text{if $n=1$}; \\ 
 \frac{(-1)^n (2n+1)(2n-1)-(2n^2+1)}{4(n-1)n}, &\text{if $n\geq 2$}.  
 \end{cases} 
\end{equation}
\end{corollary}
\begin{proof}
Set $r=2$ in \eqref{Cheb.eq2}. After simplification of the binomial coefficients this results in the relation
\begin{equation*}
\sum_{k=0}^n (-1)^{n-k} 2^{2k} \frac{\binom{n+k}{n-k}}{k+2} 
= \frac{1}{(n+2)(n+1)}\sum_{k=0}^{n} (-1)^{n-k} (n+1-k) \frac{\binom{2n+1}{2k+1}}{\binom{n}{k}}.
\end{equation*}
The left hand side equals
\begin{equation*}
\sum_{k=0}^n (-1)^{n-k} 2^{2k} \frac{\binom{n+k}{n-k}}{k+2} = 2 \int_0^1 t^3 U_{2n}(t)\, dt = \frac{(2n+1)(2n(n+1)-3)+3(-1)^n}{4(n-1)n(n+1)(n+2)}
\end{equation*}
and hence
\begin{equation*}
\sum_{k=0}^{n} (-1)^{n-k} (n+1-k) \frac{\binom{2n+1}{2k+1}}{\binom{n}{k}}
= \begin{cases}
 -5, & \text{if $n=1$}; \\ 
 \frac{(-1)^n (2n+1)(2n(n+1)-3)+3}{4(n-1)n}, &\text{if $n\geq 2$}.  
 \end{cases}
\end{equation*}
The final step is to split this sum into two parts, use the sum identity from \eqref{Cheb_binfrac} and simplify. 
\end{proof}
\begin{remark}
Similar combinatorial sums were studied recently by Chu and Guo \cite{ChuGuo}. They have proved, among other things, the evaluations
\begin{equation*}
\sum_{k=0}^{n} (-1)^{k} \frac{\binom{n+2}{2k+1}}{\binom{n}{k}} = 2 \quad\text{and}\quad 
\sum_{k=0}^{n} (-1)^{k} \frac{\binom{n+2}{2k}}{\binom{n}{k}} = -\frac{2}{n}.
\end{equation*}
\end{remark}

\begin{theorem}\label{Cheb_2}
If $s,r\in\mathbb C\setminus\mathbb Z^{-}$ such that $s\ne 0$ and $r-s\not\in\mathbb Z^{-}$, then
\begin{align}\label{Cheb_eq3}
& \sum_{k=0}^n (-1)^{k+1} 2^{2k} \frac{\binom{n+k}{n-k}}{(k+s)\binom{k+r}{k+s}} H_{k+r} \nonumber \\
&\quad = \sum_{k=0}^n (-1)^{n-k} \frac{\binom{2n+1}{2k+1}}{(n-k+s)\binom{n+r}{n-k+s}} (H_{k+r-s} - H_{n+r} - H_{r-s}).
\end{align}
\end{theorem}
\begin{proof}
Differentiate \eqref{Cheb.eq1} with respect to $r$ and simplify.
\end{proof}

\begin{corollary}
    We have
    $$\sum_{k=0}^n(-1)^{k+1}2^{2k}\frac{\binom{n+k}{n-k}}{k+r}H_{k+r}=\sum_{k=0}^n(-1)^{n-k}\frac{\binom{2n+1}{2k+1}}{(n-k+r)\binom{n+r}{k}}(H_k-H_{n+r})$$
    and 
    $$\sum\limits_{k=0}^n(-1)^{k+1}\frac{2^{2k+1}}{2k+1}\binom{n+k}{n-k}O_{k+1}\\
        =\frac{2(-1)^{n+1}}{2n+1}O_{n+1}+\sum\limits_{k=0}^n(-1)^{n-k}\frac{2^{2k}}{2k+1}\frac{\binom{n}{k}}{\binom{2k}{k}}H_k.$$
\end{corollary}
\begin{proof}
    Set $r=s$ in \eqref{Cheb_eq3}. The particular case follows from setting $r=\tfrac{1}{2}$, which yields with \eqref{eq.binomrs1}
    $$\sum\limits_{k=0}^n(-1)^{k+1}2^{2k}\frac{\binom{n+k}{n-k}}{2k+1}(2O_{k+1}-2\ln 2)=\sum\limits_{k=0}^n(-1)^{n-k}\frac{2^{2k}}{2k+1}\frac{\binom{n}{k}}{\binom{2k}{k}}(H_k-2O_{n+1}+2\ln 2).$$
    We use \eqref{eq:cor28sum} to remove logarithms. Finally, we use \eqref{eqincor:sum=1/2n+1} to simplify the right-hand side and the result follows.
\end{proof}

\section{Additional results}

In this section, we present an approach that is closely related to the general algorithm. First, we recall the following Lemma 
from \cite{Adegoke0}, which shows that the standard form \eqref{eq:id_gen_3} has additional insights to offer.

\begin{lemma}[{\cite[Corollary 3]{Adegoke0}}]\label{lem.t3lc1aa}
Let an arbitrary polynomial identity have the following form:
\begin{equation}\label{eq.a1lk6eb}
\sum_{k = s}^w f(k) (1 + t)^k = \sum_{k = m}^r g(k)\,t^k,
\end{equation}
where $m$, $w$, $r$ and $s$ are non-negative integers, $f(k)$ and $g(k)$ are sequences, and $t$ is a complex variable. 
Let $v$ be an arbitrary real number. Then
\begin{equation}\label{eq.ly7bawk}
\sum_{k = s}^w {\frac{{f(k)}}{{2^k }}\binom{{2k + v}}{{\left( {2k + v} \right)/2}}\binom{{k + v}}{{v/2}}^{ - 1} } = \sum_{k = \left\lfloor {(m + 1)/2} \right\rfloor }^{\left\lfloor {r/2} \right\rfloor } {\frac{{g(2k)}}{{2^{2k} }}\binom{{2k}}{k}\binom{{\left( {2k + v} \right)/2}}{{v/2}}^{-1}} ,
\end{equation}
and
\begin{equation}\label{eq.a1vfu08}
\sum_{k = m}^r {\frac{{g(k)( - 1)^k }}{{2^k }}\binom{{2k + v}}{{\left( {2k + v} \right)/2}}\binom{{k + v}}{{v/2}}^{ - 1}}  = \sum_{k = \left\lfloor {(s + 1)/2} \right\rfloor }^{\left\lfloor {w/2} \right\rfloor } {\frac{{ f(2k)}}{{2^{2k} }}\binom{{2k}}{k}\binom{{\left( {2k + v} \right)/2}}{{v/2}}^{ - 1} } .
\end{equation}
\end{lemma}
In particular,
\begin{equation}
\sum_{k = s}^w \frac{{f(k)}}{{2^k }}\binom{{2k}}{k} = \sum_{k = \left\lfloor {(m + 1)/2} \right\rfloor }^{\left\lfloor {r/2} \right\rfloor } {\frac{{g(2k)}}{{2^{2k} }}\binom{{2k}}{k}} ,
\end{equation}
and
\begin{equation}
\sum_{k = m}^r \frac{{g(k)( - 1)^k }}{{2^k }}\binom{{2k}}{k} = \sum_{k = \left\lfloor {(s + 1)/2} \right\rfloor }^{\left\lfloor {w/2} \right\rfloor } {\frac{{ f(2k)}}{{2^{2k} }}\binom{{2k}}{k}} .
\end{equation}

\begin{lemma}[{\cite[Theorem 8]{Adegoke0}}]\label{lem.gvw1qxa}
Let the polynomial identity be given as in~\eqref{eq.a1lk6eb}. Let $u$ and $v$ be arbitrary complex numbers such that $\Re u>-1$ 
and $\Re v>-1$. Then
\begin{equation}\label{eq.cnwt5zb}
\begin{split}
&\sum_{k = s}^n {\frac{{f(k)}}{{2^{2k} }}\binom{{v}}{{v/2}}\binom{{2k + u}}{{\left( {2k + u} \right)/2}}\binom{{\left( {2k + u + v} \right)/2}}{{v/2}}^{-1}}\\ &\qquad  = \sum_{k = m}^r {\frac{{( - 1)^{k} g(k)}}{{2^{2k} }}\binom{{u}}{{u/2}}\binom{{2k + v}}{{\left( {2k + v} \right)/2}}\binom{{\left( {2k + u + v} \right)/2}}{{u/2}}^{-1}} .
\end{split}
\end{equation}
\end{lemma}
In particular,
\begin{equation}
\sum_{k = s}^n {\frac{{f(k)}}{{2^{2k} }}\binom{{2k}}{{k}}} = \sum_{k = m}^r {\frac{{( - 1)^{k} g(k)}}{{2^{2k} }}\binom{{2k}}{{k}}} .
\end{equation}

We can now write more identities, for instance from~\eqref{eq.1}, ~\eqref{eq.3} and \eqref{eq.4}.

Identity~\eqref{eq.1} can be written in the standard form~\eqref{eq.a1lk6eb} with
\begin{equation}\label{eq.uqblgup}
f(k) = \frac{\delta_{nk} \left( {H_k + 1} \right) - 1}{n - k + \delta_{nk}},\quad g(k) = \binom{{n}}{k}H_k ,\quad w = n = r,\quad s = m = 0,
\end{equation}
where here and throughout this paper $\delta_{ij}$ is Kronecker's delta symbol having the value $1$ when $i=j$ and zero otherwise. 
Identity~\eqref{eq.3} can be written in the standard form~\eqref{eq.a1lk6eb} with
\begin{align}
f(k) &= \delta_{nk} H_n^2 - \left( {1 - \delta_{nk} } \right)\frac{{\left( {H_n - 2H_{n - k - 1 + \delta_{nk} } + H_k } \right)}}{{n - k + \delta_{nk} }},\\
g(k) &= \binom{n}{k}H_k^2,
\end{align}

We now formulate the result for \eqref{eq.4}. Similar results can be obtained for other, just mentioned identities.

\begin{theorem}
We have
\begin{equation}
\sum_{k = 0}^{\left\lfloor {n/2} \right\rfloor} \binom{{n}}{{2k}}2^{n - 2k} \binom{{2k}}{k}H_{2k}^{(2)} = \binom{{2n}}{n}H_n^{(2)} - \sum_{k = 0}^{n - 1} {\binom{{2k}}{k}2^{n - k} \frac{{H_n - H_k }}{{n - k}}}, 
\end{equation}
\begin{equation}
\begin{split}
\sum_{k = 0}^n {( - 1)^k \binom{{n}}{k}2^{ - k} \binom{{2k}}{k}H_k^{(2)} }  &= -\sum_{k = 0}^{\left\lfloor {n/2} \right\rfloor  - 1} {\binom{{2k}}{k}2^{ - 2k} \frac{{H_n - H_{2k} }}{{n - 2k}}} \\
&\qquad\qquad +  \begin{cases}
 2^{ - n} \binom{{n}}{{n/2}}H_n^{(2)},&\text{if $n$ is even;}  \\ 
  - 2^{ - n + 1} \binom{{n - 1}}{{\left( {n - 1} \right)/2}}\dfrac 1n,&\text{if $n$ is odd;} \\ 
 \end{cases} 
\end{split}
\end{equation}
and more generally, if $n$ is a non-negative integer and $v$ is a real number, then
\begin{equation}\label{eq.d0e2dlw}
\begin{split}
&\sum_{k = 0}^{\left\lfloor {n/2} \right\rfloor} \binom{n}{2k} 2^{n - 2k} \binom{2k}{k} \binom{\left( {2k + v} \right)/2}{{v/2}}^{- 1} H_{2k}^{(2)} \\
&\qquad= \binom{{2n + v}}{{\left( {2n + v} \right)/2}}\binom{{n + v}}{{v/2}}^{ - 1} H_n^{(2)}\\ 
&\qquad\qquad - \sum_{k = 0}^{n - 1} {\binom{{2k + v}}{{\left( {2k + v} \right)/2}}2^{n - k} \binom{{k + v}}{{v/2}}^{ - 1} \frac{{H_n - H_k }}{{n - k}}} ,
\end{split}
\end{equation}
and
\begin{equation}\label{eq.r66cnje}
\begin{split}
&\sum_{k = 0}^n {( - 1)^k \binom{{n}}{k}2^{ - k} \binom{{2k + v}}{{\left( {2k + v} \right)/2}}\binom{{k + v}}{{v/2}}^{ - 1} H_k^{(2)} }\\ 
&\qquad = -\sum_{k = 0}^{\left\lfloor {n/2} \right\rfloor  - 1} {\binom{{2k}}{k}2^{ - 2k} \binom{{\left( {2k + v} \right)/2}}{{v/2}}^{ - 1} \frac{{H_n - H_{2k} }}{{n - 2k}}}\\ 
&\qquad\qquad +  \begin{cases}
  \binom{{n}}{{n/2}} 2^{ - n}\binom{{\left( {n + v} \right)/2}}{{v/2}}^{ - 1} H_n^{(2)},&\text{if $n$ is even;} \\ 
  - \binom{{n - 1}}{{\left( {n - 1} \right)/2}}2^{ - n + 1} \binom{{\left( {n - 1 + v} \right)/2}}{v/2}^{-1}\dfrac 1n,&\text{if $n$ is odd.} \\ 
 \end{cases} 
\end{split}
\end{equation}
\end{theorem}
\begin{proof}
Identities~\eqref{eq.d0e2dlw} and~\eqref{eq.r66cnje} now follow from~\eqref{eq.ly7bawk} and~\eqref{eq.a1vfu08}.
\end{proof}

\begin{theorem}
We have
\begin{equation}
\sum_{k = 0}^n {( - 1)^k \binom{{n}}{k}2^{ - 2k} \binom{{2k}}{k}H_k^{(2)} }  = \binom{{2n}}{n}2^{ - 2n} H_n^{(2)}  - \sum_{k = 0}^{n - 1} {\frac{{H_n - H_k }}{{n - k}}\binom{{2k}}{k}2^{ - 2k} } ,
\end{equation}
and more generally, if $n$ is a non-negative integer and $u$ and $v$ are complex numbers such that $\Re u>-1$ and $\Re v>-1$, then
\begin{equation}
\begin{split}
&\sum_{k = 0}^n {( - 1)^k \binom{{n}}{k}2^{ - 2k} \binom{{2k + v}}{{\left( {2k + v} \right)/2}}\binom{{\left( {2k + u + v} \right)/2}}{{u/2}}^{ - 1} H_k^{(2)} } \\
&\qquad = \binom{{v}}{{v/2}}\binom{{u}}{{u/2}}^{ - 1} \binom{{2n + u}}{{\left( {2n + u} \right)/2}}\binom{{\left( {2n + u + v} \right)/2}}{{v/2}}^{ - 1} 2^{ - 2n} H_n^{(2)}\\ 
&\qquad\qquad - \binom{{v}}{{v/2}}\binom{{u}}{{u/2}}^{ - 1} \sum_{k = 0}^{n - 1} {\frac{{H_n - H_k }}{{n - k}}\binom{{2k + u}}{{\left( {2k + u} \right)/2}}2^{ - 2k} \binom{{\left( {2k + u + v} \right)/2}}{{v/2}}^{ - 1} } .
\end{split}
\end{equation}
\end{theorem}
\begin{proof}
Use identity~\eqref{eq.cnwt5zb} with the parameters and variables given in~\eqref{eq.uqblgup}. Particular cases follow after setting $u=v$ and then replacing $u$ by $2u$ and then setting $u=0$.
\end{proof}

\section{Conclusion}\label{Section_Conclusion}

In this article, we have introduced a general framework (algorithm) to study combinatorial sums associated with certain classes of polynomial identities. The sums involve products of binomial coefficients and also harmonic numbers. To underline its usefulness we have shown applications based on some prominent polynomial identities from the recent past. This is only the starting point and many new families of sums can be studied. We invite the readers to apply the framework to sums associated with
\begin{itemize}
\item the identity of Simons \cite{Simons}
\begin{equation*}
\sum_{k=0}^n \binom{n}{k} \binom{n+k}{k} t^k = \sum_{k=0}^n (-1)^{n+k} \binom{n}{k} \binom{n+k}{k} (1+t)^k
\end{equation*}
\item the Narayana polynomials \cite{Mansour1,Mansour2}
\begin{equation*}
\sum_{k=0}^n \frac{1}{n} \binom{n}{k-1} \binom{n}{k} t^k = \frac{1}{n+1} \sum_{k=0}^n (-1)^{k} \binom{n+1}{k} \binom{2n-k}{n} (1-t)^k
\end{equation*} 
\item an identity involving central binomial coefficients \cite{Gould1}
\begin{equation*}
\sum_{k=0}^n 2^{2(n-k)} \binom{n}{k} \binom{2k}{k} t^k = \sum_{k=0}^n \binom{2k}{k} \binom{2(n-k)}{n-k} (1+t)^k
\end{equation*} 
\end{itemize} 
and many more.

\end{document}